\documentclass[12pt,reqno]{article}

\usepackage[usenames]{color}
\usepackage[colorlinks=true,
linkcolor=webgreen, filecolor=webbrown,
citecolor=webgreen]{hyperref}

\definecolor{webgreen}{rgb}{0,.5,0}
\definecolor{webbrown}{rgb}{.6,0,0}

\usepackage{amssymb}
\usepackage{graphicx}
\usepackage{amscd}
\usepackage{lscape}
\usepackage{tikz}
\usepackage{tikz-cd}
\usepackage{pgfplots}

\usetikzlibrary{matrix}
\usetikzlibrary{fit,shapes}
\usetikzlibrary{positioning, calc}
\tikzset{circle node/.style = {circle,inner sep=1pt,draw, fill=white},
        X node/.style = {fill=white, inner sep=1pt},
        dot node/.style = {circle, draw, inner sep=5pt}
        }
\usepackage{tkz-fct}

\usepackage{amsthm}
\newtheorem{theorem}{Theorem}
\newtheorem{lemma}[theorem]{Lemma}
\newtheorem{proposition}[theorem]{Proposition}
\newtheorem{corollary}[theorem]{Corollary}

\theoremstyle{definition}

\newtheorem{example}[theorem]{Example}

\usepackage{float}

\usepackage{graphics,amsmath}
\usepackage{amsfonts}
\usepackage{latexsym}
\usepackage{epsf}

\DeclareMathOperator{\Rev}{Rev}
\DeclareMathOperator{\revert}{revert}

\newcommand{\seqnum}[1]{\href{http://oeis.org/#1}{\underline{#1}}}

\begin{document}

\begin{center}
\vskip 1cm{\LARGE\bf On the inversion of Riordan arrays} \vskip 1cm \large
Paul Barry\\
School of Science\\
Waterford Institute of Technology\\
Ireland\\
\href{mailto:pbarry@wit.ie}{\tt pbarry@wit.ie}
\end{center}
\vskip .2 in

\begin{abstract} Many Riordan arrays play a significant role in algebraic combinatorics. We explore the inversion of Riordan arrays in this context. We give a general construct for the inversion of a Riordan array, and study this in the  case of various subgroups of the Riordan group. For instance, we show that the inversion of an ordinary Bell matrix is an exponential Riordan array in the associated subgroup. Examples from combinatorics and algebraic combinatorics illustrate the usefulness of such inversions. We end with a brief look at the inversion of exponential Riordan arrays. A final example places Airey's convergent factor in the context of a simple exponential Riordan array.
\end{abstract}

\section{Preliminaries}
We let $\mathcal{F}=\{ a_0+a_1 x+ a_2 x^2+ \cdots \,| a_i \in \mathbf{R}\}$ be the set of formal power series with coefficients $a_i$ drawn from the ring $\mathbf{R}$. This ring can be any ring over which the operations we will carry out make sense, but for concreteness it can be assumed to be $\mathbb{Q}$. For combinatorial problems, the ring $\mathbf{R}\}$ is often the ring of integers $\mathbb{Z}$. It will be seen that most of the matrices we deal with have integer entries. We shall use two distinguished subsets of $\mathcal{F}$, namely
$$\mathcal{F}_0=\{ a_0+a_1 x+ a_2 x^2+ \cdots \,| a_i \in \mathbf{R}, a_0 \ne 0\},$$ and
$$\mathcal{F}_1=\{ a_1 x+ a_2 x^2+ \cdots \,| a_i \in \mathbf{R}, a_1 \ne 0\}.$$ Elements of $\mathcal{F}_1$ are composable and possess compositional inverses. If $f(x) \in \mathcal{F}_1$, we shall denote by $\bar{f}(x)$ or $\Rev(f)(x)$ its compositional inverse. Thus we have $\bar{f}(f(x))=x$ and $f(\bar{f}(x))=x$.

Throughout our exposition, we will stipulate that $g(x) \in \mathcal{F}_0$, $f(x) \in \mathcal{F}_0$, $u(x) \in \mathcal{F}_0$ and $v(x) \in \mathcal{F}_1$.
Without much loss of generality, we will also assume that $g(0)=f(0)=u(0)=v'(0)=1$.

A Riordan array \cite{book, SGWW} may be defined by a pair $(u(x), v(x)) \in \mathcal{F}_0 \times  \mathcal{F}_1$ of power series, represented by the invertible lower-triangular matrix $\left(t_{n,k}\right)$ where
$$t_{n,k}=[x^n] u(x)v(x)^k.$$
Here, the functional $[x^n]$ acts on elements of $\mathcal{F}$ by returning the coefficient of $x^n$ of the power series in question \cite{MC}. We shall sometimes write $(g, f)_{n,k}$ for $t_{n,k}$. 

The product of two Riordan arrays is again a Riordan array, defined by 
$$(g, f) \cdot (u, v)=(g.u(f), v(f)).$$ In the matrix representation, this corresponds to the usual multiplication of matrices. The inverse of a Riordan array is a Riordan array, given by 
$$(g, f)^{-1}=\left(\frac{1}{g(\bar{f})}, \bar{f}\right).$$ In the matrix representation, this coincides with the normal inverse of a matrix. The identity element is $(1,x)$.

Note that all matrices considered will be lower-triangular matrices indexed by $(n,k)$ where $0 \le n,k le \infty$. We use suitable truncations of such matrices in the text. Many examples of Riordan arrays can be found in the On-Line Encyclopedia of Integer Sequences (OEIS) \cite{SL1, SL2}, along with many of the sequences occurring in this note. Such sequences are referred to by their OEIS $Annnnnn$ numbers. For instance, the Catalan numbers $C_n=\frac{1}{n1+1} \binom{2n}{n}$ have the OEIS number \seqnum{A000108}. All the matrices in this note are assumed to be lower-triangular, and most are invertible.

The Fundamental Theorem of Riordan arrays (FTRA) specifies how a Riordan array $(g, f)$ operates on a generating function. Thus we have 
$$(g(x), f(x))\cdot h(x)= g(x)h(f(x)).$$ In the matrix representation, this corresponds to the matrix $(t_{n,k})$ multiplying the vector whose elements are $(h_0, h_1, h_2, \ldots)$ where $h(x)=\sum_{n=0}^{\infty} h_n x^n$. 

The ability to switch between the power series view of Riordan arrays, and the FTRA, on the one hand, and the linear algebra view involving matrix multiplication and matrix inverses, makes Riordan arrays a powerful tool in many of the settings of algebraic combinatorics.

To each Riordan array $(g(x), f(x))$ we associate the bivariate generating function
$$ G(x,y)=\frac{g(x)}{1-y f(x)}.$$ Thus we have that
$$G(x,y)=\sum_{n,k \ge 0} t_{n,k}x^n y^k.$$
Fixing $y$, we can then seek the $x$-inversion $\Rev_x(x G(x,y)$ of the generating function $xG(x,y)$. By the \emph{inversion of the Riordan array $(g(x), f(x))$} we shall mean the array given by the expansion of the generating function
$$\frac{1}{x}\Rev_x(x G(x,y).$$ Note that this is often called the ``REVERT'' transform of $G$. Where no confusion will be caused, we will continue to call it the inversion of $G$, or of the corresponding Riordan array. We shall denote this array by $$(g(x), f(x))^!$$ It is a Riordan array only in special cases.
We give two combinatorially important examples.
\begin{example} We consider the Riordan array $\left(\frac{1}{1+x}, -x\right)$ which begins
$$\left(
\begin{array}{cccccc}
 1 & 0 & 0 & 0 & 0 & 0 \\
 -1 & -1 & 0 & 0 & 0 & 0 \\
 1 & 1 & 1 & 0 & 0 & 0 \\
 -1 & -1 & -1 & -1 & 0 & 0 \\
 1 & 1 & 1 & 1 & 1 & 0 \\
 -1 & -1 & -1 & -1 & -1 & -1 \\
\end{array}
\right).$$
To obtain its inversion, we solve the equation
$$\frac{\frac{u}{1+u}}{1-y(-u)}=x,$$ or
$$\frac{u}{(1+u)(1+yu)}=x$$ for $u$. We obtain
$$u=\frac{1-(1+y)x-\sqrt{1-2(1+y)x+(1-y)^2x^2}}{2xy}.$$ Thus the inversion of $G(x,y)$ in this case is given by
$$\frac{u}{x}=\frac{1-(1+y)x-\sqrt{1-2(1+y)x+(1-y)^2x^2}}{2x^2y}.$$
This expands to give the number triangle that begins
$$\left(
\begin{array}{cccccc}
 1 & 0 & 0 & 0 & 0 & 0 \\
 1 & 1 & 0 & 0 & 0 & 0 \\
 1 & 3 & 1 & 0 & 0 & 0 \\
 1 & 6 & 6 & 1 & 0 & 0 \\
 1 & 10 & 20 & 10 & 1 & 0 \\
 1 & 15 & 50 & 50 & 15 & 1 \\
\end{array}
\right).$$
This is the triangle of Narayana numbers $\left(\frac{1}{k+1}\binom{n}{k}\binom{n+1}{k}\right)$ \seqnum{A001263}. Thus we have 
$$\left(
\begin{array}{cccccc}
 1 & 0 & 0 & 0 & 0 & 0 \\
 -1 & -1 & 0 & 0 & 0 & 0 \\
 1 & 1 & 1 & 0 & 0 & 0 \\
 -1 & -1 & -1 & -1 & 0 & 0 \\
 1 & 1 & 1 & 1 & 1 & 0 \\
 -1 & -1 & -1 & -1 & -1 & -1 \\
\end{array}
\right)^!=\left(
\begin{array}{cccccc}
 1 & 0 & 0 & 0 & 0 & 0 \\
 1 & 1 & 0 & 0 & 0 & 0 \\
 1 & 3 & 1 & 0 & 0 & 0 \\
 1 & 6 & 6 & 1 & 0 & 0 \\
 1 & 10 & 20 & 10 & 1 & 0 \\
 1 & 15 & 50 & 50 & 15 & 1 \\
\end{array}
\right).$$
\end{example}
\begin{example} In this second example, we consider the Riordan array $\left(\frac{1}{(1+x)^2},\frac{-x}{1+x}\right)$, which begins
$$\left(
\begin{array}{cccccc}
 1 & 0 & 0 & 0 & 0 & 0 \\
 -2 & -1 & 0 & 0 & 0 & 0 \\
 3 & 3 & 1 & 0 & 0 & 0 \\
 -4 & -6 & -4 & -1 & 0 & 0 \\
 5 & 10 & 10 & 5 & 1 & 0 \\
 -6 & -15 & -20 & -15 & -6 & -1 \\
\end{array}
\right).$$
Apart from signs, this matrix enumerates faces of the $n$-simplex.
In order to find its inversion, we solve for $u$ the equation
$$\frac{\frac{u}{(1+u)^2}}{1-y \frac{-u}{1+u}}=x,$$ or
$$\frac{u}{(1+u)(1+(1+y)u)}=x.$$ Choosing the solution $u$ with $u(0)=0$, we obtain that
$$\frac{u}{x}=\frac{1-(y+2)x-\sqrt{1-2(y+2)x+x^2y^2}}{2(1+y)x^2}.$$
This is the generating function of the number triangle \seqnum{A126216} that begins
$$\left(
\begin{array}{cccccc}
 1 & 0 & 0 & 0 & 0 & 0 \\
 2 & 1 & 0 & 0 & 0 & 0 \\
 5 & 5 & 1 & 0 & 0 & 0 \\
 14 & 21 & 9 & 1 & 0 & 0 \\
 42 & 84 & 56 & 14 & 1 & 0 \\
 132 & 330 & 300 & 120 & 20 & 1 \\
\end{array}
\right).$$
Among many combinatorial interpretations, this matrix enumerates faces of the $n$-associahedron. In the theory of Koszul operads \cite{Loday}, this result follows from the fact that
$$ Trias^!=Tridend.$$
\end{example}

\section{Generalities}
In this section, we gather some facts that will be important in later section. We begin by recalling the definitions of two transforms of importance. The \emph{binomial matrix} is the matrix $\mathbf{B}=\left(\binom{n}{k}\right)_{0 \le k,n \le n}$ \seqnum{A007318}. This is the Riordan array $\left(\frac{1}{1-x}, \frac{x}{1-x}\right)$. It begins
$$\left(
\begin{array}{cccccc}
 1 & 0 & 0 & 0 & 0 & 0 \\
 1 & 1 & 0 & 0 & 0 & 0 \\
 1 & 2 & 1 & 0 & 0 & 0 \\
 1 & 3 & 3 & 1 & 0 & 0 \\
 1 & 4 & 6 & 4 & 1 & 0 \\
 1 & 5 & 10 & 10 & 5 & 1 \\
\end{array}
\right).$$

The \emph{binomial transform} of a sequence (regarded as an infinite vector) or an array is the result of multiplying that vector or matrix by the binomial matrix $\mathbf{B}$. In terms of generating functions, the Fundamental Theorem of Riordan arrays tells us that if such a sequence has generating function $g(x)$, then its binomial transform has generating function $\frac{1}{1-x}g\left(\frac{x}{1-x}\right)$. Similarly for an array with bivariate generating function $G(x,y)$, its ($x$-)binomial transform will have generating function $\frac{1}{1-x}G\left(\frac{x}{1-x},y\right)$.

The \emph{inverse binomial transform} is obtained by multiplying by $\mathbf{B}^{-1}$, which corresponds to the Riordan array $\left(\frac{1}{1+x}, \frac{x}{1+x}\right)$. In terms of generating functions, this means that the inverse binomial transform of the sequence with generating function $g(x)$ has generating function $\frac{1}{1+x} g\left(\frac{x}{1+x}\right)$.

The \emph{invert transform} of a sequence with generating function $g(x) \in \mathcal{F}_0$ is the sequence with generating function $\frac{g(x)}{1-xg(x)}$. More generally, we shall say the $\frac{g(x)}{1- \alpha x g(x)}$ is the invert$(\alpha)$ transform of $g(x)$.

These two transforms are related by the process of inversion.
\begin{proposition} The revert transform of the invert transform of $g(x) \in \mathcal{F}_0$ is the inverse binomial transform of the revert transform of $g(x)$.
\end{proposition}
\begin{proof} We recall that the revert transform of $g(x)$ is given by
$$\frac{1}{x} \Rev(x g(x)).$$
We let $v(x)=xg(x)$ and we write $\bar{v}(x)=\Rev(xg(x))$.
Then the inverse binomial transform of the revert transform of $g(x)$ is given by
$$\frac{1}{1+x}\left(\frac{1}{x}\bar{v}\right)\left(\frac{x}{1+x}\right).$$
We then have
\begin{align*}
\frac{1}{1+x}\left(\frac{1}{x}\bar{v}\right)\left(\frac{x}{1+x}\right)&=\frac{1}{1+x}\frac{1}{\frac{x}{1+x}}\bar{v}\left(\frac{x}{1+x}\right)\\
&=\frac{1}{x} \bar{v}\left(\frac{x}{1+x}\right)\\
&=\frac{1}{x} \bar{v} \circ \left(\frac{x}{1+x}\right)(x)\\
&=\frac{1}{x} \bar{v} \circ \overline{\frac{x}{1-x}}(x)\\
&= \frac{1}{x} \overline{\frac{x}{1-x} \circ v}(x)\\
&=\frac{1}{x} \Rev\left(\frac{v(x)}{1-v(x)}\right)\\
&=\frac{1}{x} \Rev\left(\frac{xg(x)}{1-xg(x)}\right).\end{align*}
\end{proof}

We can extend the notion of invert transform to Riordan arrays by operating on the $x$-variable. Thus the invert transform of the Riordan array is the array with generating function
$$\frac{G(x,y)}{1-x G(x,y)},\quad \text{where} \quad G(x,y)=\frac{g(x)}{1-yf(x)}.$$ We have the following result.
\begin{proposition} The invert transform of the Riordan array $(g(x), f(x))$ is the Riordan array
$$ \left(\frac{g(x)}{1-xg(x)}, \frac{f(x)}{1-x g(x)}\right).$$
\end{proposition}
\begin{proof} The Riordan array $(g(x), f(x))$ has bivariate generating function $\frac{g(x)}{1-yf(x)}$.

The ($x$-)invert transform of this is given by
$$\frac{\frac{g(x)}{1-yf(x)}}{1-x\frac{g(x)}{1-yf(x)}}=\frac{g(x)}{1-x g(x)- y f(x)}.$$
The Riordan array $\left(\frac{g(x)}{1-xg(x)}, \frac{f(x)}{1-x g(x)}\right)$ has a bivariate generating function given by
$$\frac{\frac{g(x)}{1-xg(x)}}{1-y\frac{f(x)}{1-x g(x)}}=\frac{g(x)}{1-x g(x)- y f(x)}.$$
\end{proof}
More generally, the invert$(\alpha)$ transform of the Riordan array $(g(x), f(x))$ is the Riordan array
$$ \left(\frac{g(x)}{1-\alpha xg(x)}, \frac{f(x)}{1-\alpha x g(x)}\right).$$
\begin{example} We consider the Riordan array $\left(\frac{1}{1+x},-x\right)$ of our first example. The invert transform of this array is given by
$$\left(\frac{\frac{1}{1+x}}{1-x \frac{1}{1+x}}, \frac{-x}{1-x \frac{1}{1+x}}\right)=(1, -x(1+x)).$$
This is the array $\left((-1)^k \binom{k}{n-k}\right)$, which begins
$$\left(
\begin{array}{cccccc}
 1 & 0 & 0 & 0 & 0 & 0 \\
 0 & -1 & 0 & 0 & 0 & 0 \\
 0 & -1 & 1 & 0 & 0 & 0 \\
 0 & 0 & 2 & -1 & 0 & 0 \\
 0 & 0 & 1 & -3 & 1 & 0 \\
 0 & 0 & 0 & -3 & 4 & -1 \\
\end{array}
\right).$$
To obtain its inversion, we solve the equation
$$\frac{u}{1+yu(1+u)}=x.$$ We obtain that
$$\frac{u}{x}=\frac{1-xy-\sqrt{1-2xy+(y-4)x^2y}}{2x^2y}.$$
This expands to give the array that begins
$$\left(
\begin{array}{cccccc}
 1 & 0 & 0 & 0 & 0 & 0 \\
 0 & 1 & 0 & 0 & 0 & 0 \\
 0 & 1 & 1 & 0 & 0 & 0 \\
 0 & 0 & 3 & 1 & 0 & 0 \\
 0 & 0 & 2 & 6 & 1 & 0 \\
 0 & 0 & 0 & 10 & 10 & 1 \\
\end{array}
\right).$$ This is the array $\left(\binom{n}{2(n-k)}C_{n-k}\right)$.
We now note that we have
$$\left(
\begin{array}{cccccc}
 1 & 0 & 0 & 0 & 0 & 0 \\
 1 & 1 & 0 & 0 & 0 & 0 \\
 1 & 2 & 1 & 0 & 0 & 0 \\
 1 & 3 & 3 & 1 & 0 & 0 \\
 1 & 4 & 6 & 4 & 1 & 0 \\
 1 & 5 & 10 & 10 & 5 & 1 \\
\end{array}
\right)\cdot \left(
\begin{array}{cccccc}
 1 & 0 & 0 & 0 & 0 & 0 \\
 0 & 1 & 0 & 0 & 0 & 0 \\
 0 & 1 & 1 & 0 & 0 & 0 \\
 0 & 0 & 3 & 1 & 0 & 0 \\
 0 & 0 & 2 & 6 & 1 & 0 \\
 0 & 0 & 0 & 10 & 10 & 1 \\
\end{array}
\right)
=\left(
\begin{array}{cccccc}
 1 & 0 & 0 & 0 & 0 & 0 \\
 1 & 1 & 0 & 0 & 0 & 0 \\
 1 & 3 & 1 & 0 & 0 & 0 \\
 1 & 6 & 6 & 1 & 0 & 0 \\
 1 & 10 & 20 & 10 & 1 & 0 \\
 1 & 15 & 50 & 50 & 15 & 1 \\
\end{array}
\right).$$
\end{example}
The generating function of the reversal of an array with generating function $G(x,y)$ is given by $G\left(xy,\frac{1}{y}\right)$. We can deduce from this that the inversion of the reversal of an array is the reversal of the inversion of the original array.
Thus the inversion of the array that begins
$$\left(
\begin{array}{cccccc}
 1 & 0 & 0 & 0 & 0 & 0 \\
 -1 & -2 & 0 & 0 & 0 & 0 \\
 1 & 3 & 3 & 0 & 0 & 0 \\
 -1 & -4 & -6 & -4 & 0 & 0 \\
 1 & 5 & 10 & 10 & 5 & 0 \\
 -1 & -6 & -15 & -20 & -15 & -6 \\
\end{array}
\right)$$ is given by the array that begins
$$\left(
\begin{array}{cccccc}
 1 & 0 & 0 & 0 & 0 & 0 \\
 1 & 2 & 0 & 0 & 0 & 0 \\
 1 & 5 & 5 & 0 & 0 & 0 \\
 1 & 9 & 21 & 14 & 0 & 0 \\
 1 & 14 & 56 & 84 & 42 & 0 \\
 1 & 20 & 120 & 300 & 330 & 132 \\
\end{array}
\right).$$
The initial column of the array with generating function $G(x,y)$ has generating function $G(x,0)$. The initial column of the inversion of the array is then the revert transform of this original initial column. Likewise, the row sums of the array with generating function $G(x,y)$ have generating function $G(x,1)$. We then have that the row sums of the inversion of an array are given by the revert transform of the row sums of the original array.
\begin{example} The revert transform of the sequence
$$1,-2,3,-4,5,-6,\ldots$$ is the sequence
$$1,2,5,14,42,132,\ldots.$$
The row sums of the matrix that begins
$$\left(
\begin{array}{cccccc}
 1 & 0 & 0 & 0 & 0 & 0 \\
 -2 & -1 & 0 & 0 & 0 & 0 \\
 3 & 3 & 1 & 0 & 0 & 0 \\
 -4 & -6 & -4 & -1 & 0 & 0 \\
 5 & 10 & 10 & 5 & 1 & 0 \\
 -6 & -15 & -20 & -15 & -6 & -1 \\
\end{array}
\right)$$ begin
$$1, -3, 7, -15, 31, -63, 127,\ldots.$$
The revert transform of this sequence, which begins
$$1, 3, 11, 45, 197, 903,\ldots$$ is then given by the row sums of the array that begins
$$\left(
\begin{array}{cccccc}
 1 & 0 & 0 & 0 & 0 & 0 \\
 2 & 1 & 0 & 0 & 0 & 0 \\
 5 & 5 & 1 & 0 & 0 & 0 \\
 14 & 21 & 9 & 1 & 0 & 0 \\
 42 & 84 & 56 & 14 & 1 & 0 \\
 132 & 330 & 300 & 120 & 20 & 1 \\
\end{array}\right).$$
\end{example}

As the inversion process is involutive, we note that the inversion of the inversion of an array is the original array.

Of importance in the sequel will be the technique of Lagrange inversion \cite{Henrici, When}. The form that we shall use is the following.  We have

$$[x^n] H(\bar{f})= \frac{1}{n}[x^{n-1}]H'(x)\left(\frac{x}{f}\right)^n,$$
where $H(x) \in \mathbf{R}[[x]]$.

\section{Main results}
We first find a general expression for the $(n,k)$-th term of the inversion $(g(x),f(x))^!$ of the Riordan array $(g(x), f(x))$. \begin{lemma} The $(n,k)$-th term $\hat{t}_{n,k}$ of the inversion of the Riordan array $(g(x), f(x))$ is given by
$$\hat{t}_{n,k}=\frac{(-1)^k}{n+1}\binom{n+1}{k}[x^n]f(x)^k \left(\frac{1}{g(x)}\right)^{n+1}.$$
\end{lemma}
\begin{proof}
Using Lagrange inversion, we find that
\begin{align*}
[x^n y^k]\frac{1}{x} \Rev\left(\frac{xg(x)}{1-yf(x)}\right)&=
[x^{n+1}y^k] \Rev\left(\frac{xg(x)}{1-yf(x)}\right)\\
&=\frac{1}{n+1}[x^n y^k]\left(\frac{1-yf(x)}{g(x)}\right)^{n+1}\\
&=\frac{1}{n+1}[x^n y^k]\sum_{j=0}^{n+1}\binom{n+1}{j}(-1)^j y^j f(x)^j \left(\frac{1}{g(x)}\right)^{n+1}\\
&=\frac{1}{n+1}[x^n] \binom{n+1}{k}f(x)^k \left(\frac{1}{g(x)}\right)^{n+1}\\
&=\frac{(-1)^k}{n+1}\binom{n+1}{k}[x^n] f(x)^k \left(\frac{1}{g(x)}\right)^{n+1}.\end{align*}
\end{proof}
\begin{proposition} The $(n,k)$-th term $\hat{t}_{n,k}$ of the inversion of the Riordan array $(g(x), f(x))$ is given by
$$\hat{t}_{n,k}=\frac{(-1)^k}{n+1}\binom{n+1}{k} \left(\left(\Rev(xg(x))\right)', f(\Rev(xg(x)))\right)_{n,k}.$$
\end{proposition}
\begin{proof} Using Lagrange inversion, we have the following chain of equalities.
\begin{align*}
[x^{n-1}]f(x)^k\left(\frac{1}{g(x)}\right)^n&=n.\frac{1}{n}[x^{n-1}] f(x)^k\left(\frac{x}{xg(x)}\right)^n\\
&=n.\frac{1}{n}[x^{n-1}]H'(x)\left(\frac{x}{xg(x)}\right)^n \quad \text{where }\,H'(x)=f(x)^k\\
&=n [x^n] H(\Rev(xg(x))\\
&=[x^{n-1}] \frac{d}{dx} H(\Rev(xg(x)))\\
&=[x^{n-1}] H'(\Rev(xg(x))).(\Rev(xg(x)))'\\
&=[x^{n-1}] (f(\Rev(xg(x))))^k.(\Rev(xg(x)))' \end{align*}
Thus we have
$$[x^n]f(x)^n \left(\frac{1}{g(x)}\right)^{n+1}=[x^n](\Rev(xg(x))'.(f(\Rev(xg(x))))^k=\left((\Rev(xg(x)))', f(\Rev(xg(x))\right)_{n,k}.$$
\end{proof}
Notice that we have the factorization
$$\left((\Rev(xg(x)))', f(\Rev(xg(x))\right)=\left((\Rev(xg(x)))', \Rev(xg(x)\right)\cdot (1, f(x)).$$

We now look at the consequences of this result for three particular subgroups of the Riordan group. These are
\begin{enumerate}
\item The \emph{Appell subgroup} of Riordan arrays of the form $(g(x),x)$.
\item The \emph{associated or Lagrange subgroup} of Riordan arrays of the form $(1, f(x))$.
\item The \emph{Bell subgroup} of Riordan arrays of the form $(g(x), xg(x))$.
\end{enumerate}
We note that Riordan arrays of the form $(f', f)$ constitute the \emph{derivative subgroup} of the Riordan group. Thus the Riordan array $\left((\Rev(xg(x)))', f(\Rev(xg(x))\right)$ is a product of an element of the derivative subgroup times an element of the associated group. 
\begin{corollary} The inversion of the Appell array $(g(x), x)$ has its general $(n,k)$-th term given by
$$\hat{t}_{n,k}=\frac{(-1)^k}{n+1}\binom{n+1}{k} \left(\left(\Rev(xg(x))\right)', \Rev(xg(x))\right)_{n,k}.$$
We have
$$\left(\left(\Rev(xg(x))\right)', \Rev(xg(x))\right)=\left((xg(x))', xg(x)\right)^{-1}.$$
\end{corollary}
\begin{corollary} The $(n,k)$-th element of the inversion of the Lagrange array $(1, f(x))=(t_{n,k})$ is given by
$$\hat{t}_{n,k}=\frac{(-1)^k}{n+1} \binom{n+1}{k}t_{n,k}.$$
\end{corollary}
This follows since $g(x)=1$ means that the array $((xg(x))', xg(x))=(1,x)$ and so its inverse matrix is also the identity matrix.
In order to characterize the inversion of a Bell matrix, we need to review the definition of an exponential Riordan array. Such an array $[u, v]$ is given by two power series $(u, v)$ drawn respectively from
$$\mathcal{F}_0^e=\{ u_0+u_1 \frac{x}{1!}+ u_2 \frac{x^2}{2!}+ \cdots \,| u_i \in \mathbf{R}, u_0 \ne 0\},$$ and
$$\mathcal{F}_1^e=\{ v_1 \frac{x}{1!}+ v_2 \frac{x^2}{2!}+ \cdots \,| v_i \in \mathbf{R}, v_1 \ne 0\}.$$
The $(n,k)$-th element of the exponential Riordan array $[u, v]$ is then given by
$$\frac{n!}{k!}[x^n] u(x)v(x)^k.$$
Given an ordinary power series $g(x) \in \mathcal{F}_0$, with $g(x)=a_0+a_1 x + a_2x^2+\cdots$, we shall write
$$g_e(x)=a_0+a_1 \frac{x}{1!} + a_2 \frac{x^2}{2!} + \cdots.$$  Thus $g_e(x) \in \mathcal{F}_0^e$. (We have $g_e(t)=\mathcal{L}^{-1}\left\{\frac{1}{s}g\left(\frac{1}{s}\right)\right\}(t)$, where $\mathcal{L}$ is the Laplace transform).
\begin{proposition}
The inversion of the Bell matrix $(g(x), xg(x))$ is given by the exponential Riordan array
$$ \left[\left(\frac{1}{x}\Rev(xg(x))\right)_e, -x\right].$$
\end{proposition}
The expression $\frac{1}{x}\Rev(xg(x))$ is the revert transform of $g(x)$, and hence we can also write this as
$$\left[(\revert(g(x))_e, -x\right].$$
\begin{proof}
The $(n,k)$-th element of the exponential Riordan array $\left[\left(\frac{1}{x}\Rev(xg(x))\right)_e, -x\right]$ is given by
$$\frac{n!}{k!} [x^n] \left(\frac{1}{x}\Rev(xg)\right)_e (-x)^k.$$ We then have
\begin{align*}
\frac{n!}{k!} [x^n] \left(\frac{1}{x}\Rev(xg)\right)_e (-x)^k&=(-1)^k\frac{n!}{k!}[x^{n-k}] \left(\frac{1}{x}\Rev(xg)\right)_e\\
&=(-1)^k\frac{n!}{k!} [x^{n-k}] \sum_{j=0}^{\infty} [x^j]\left(\frac{1}{x}\Rev(xg)\right)\frac{x^j}{j!}\\
&=(-1)^k\frac{n!}{k!} \frac{1}{(n-k)!} [x^{n-k}] \left(\frac{1}{x}\Rev(xg)\right)\\
&= (-1)^k\binom{n}{k} [x^{n-k}] \left(\frac{1}{x}\Rev(xg)\right)\\
&= (-1)^k\binom{n}{k} [x^{n-k+1}] \Rev(xg)\\
&= (-1)^k \binom{n}{k} \frac{1}{n-k+1} [x^{n-k}] \left(\frac{1}{g(x)}\right)^{n-k+1}\\
&= \frac{(-1)^k}{n+1} \binom{n+1}{k} [x^n]x^k g(x)^k \left(\frac{1}{g(x)}\right)^{n+1}\\
&= \frac{(-1)^k}{n+1} \binom{n+1}{k} [x^n](xg(x))^k \left(\frac{1}{g(x)}\right)^{n+1}.\end{align*}
This last expression is the $(n,k)$-th element of the inversion of $(g(x), xg(x))$.
\end{proof}
This is an interesting example of where an ordinary Riordan array is directly linked to an exponential Riordan array. 
\begin{example} \textbf{Chebyshev polynomials of the second kind}. The coefficient array of the scaled Chebyshev polynomials of the second kind $U_n(x/2)$ is given by the Bell matrix $\left(\frac{1}{1+x^2}, \frac{x}{1+x^2}\right)$ \cite{Classical}. This matrix \seqnum{A049310} begins
$$\left(
\begin{array}{ccccccc}
 1 & 0 & 0 & 0 & 0 & 0 & 0 \\
 0 & 1 & 0 & 0 & 0 & 0 & 0 \\
 -1 & 0 & 1 & 0 & 0 & 0 & 0 \\
 0 & -2 & 0 & 1 & 0 & 0 & 0 \\
 1 & 0 & -3 & 0 & 1 & 0 & 0 \\
 0 & 3 & 0 & -4 & 0 & 1 & 0 \\
 -1 & 0 & 6 & 0 & -5 & 0 & 1 \\
\end{array}
\right).$$ The inverse of this matrix is the matrix \seqnum{A053121} that begins 
$$\left(
\begin{array}{ccccccc}
 1 & 0 & 0 & 0 & 0 & 0 & 0 \\
 0 & 1 & 0 & 0 & 0 & 0 & 0 \\
 1 & 0 & 1 & 0 & 0 & 0 & 0 \\
 0 & 2 & 0 & 1 & 0 & 0 & 0 \\
 2 & 0 & 3 & 0 & 1 & 0 & 0 \\
 0 & 5 & 0 & 4 & 0 & 1 & 0 \\
 5 & 0 & 9 & 0 & 5 & 0 & 1 \\
\end{array}
\right).$$ 
By the result above, we have 
$$\left(\frac{1}{1+x^2}, \frac{x}{1+x^2}\right)^{!}=\left[\frac{I_1(2x)}{x}, -x\right].$$ 
This inversion matrix begins 
$$\left(
\begin{array}{ccccccc}
 1 & 0 & 0 & 0 & 0 & 0 & 0 \\
 0 & -1 & 0 & 0 & 0 & 0 & 0 \\
 1 & 0 & 1 & 0 & 0 & 0 & 0 \\
 0 & -3 & 0 & -1 & 0 & 0 & 0 \\
 2 & 0 & 6 & 0 & 1 & 0 & 0 \\
 0 & -10 & 0 & -10 & 0 & -1 & 0 \\
 5 & 0 & 30 & 0 & 15 & 0 & 1 \\
\end{array}
\right).$$
In absolute value, this matrix is \seqnum{A097610}, which counts Motzkin paths of length $n$ with $k$ horizontal steps.
\end{example}
Because of the reversibility of the inversion process, we can also start with an exponential Riordan matrix of the form $[ g_e(x), -x]$ and derive the (ordinary) Bell matrix for which it is the inversion.
\begin{example} We start with the exponential Riordan array $\left[\frac{1}{1-x}, -x\right]$. This matrix begins $$\left(
\begin{array}{cccccc}
 1 & 0 & 0 & 0 & 0 & 0 \\
 1 & -1 & 0 & 0 & 0 & 0 \\
 2 & -2 & 1 & 0 & 0 & 0 \\
 6 & -6 & 3 & -1 & 0 & 0 \\
 24 & -24 & 12 & -4 & 1 & 0 \\
 120 & -120 & 60 & -20 & 5 & -1 \\
\end{array}
\right).$$ 
This is a signed version of \seqnum{A094587}, which counts permutations on $n$ letters with exactly $k+1$ cycles and with the first $k+1$ letters in separate cycles. Reversing our steps from above, we find that the Riordan array for which this is the inversion is the Riordan array
$$\left(\sum_{n=0}^{\infty} n!x^n, x \sum_{n=0}^{\infty} n!x^n\right)^{-1}.$$ 
This array begins 
$$\left(
\begin{array}{cccccc}
 1 & 0 & 0 & 0 & 0 & 0 \\
 -1 & 1 & 0 & 0 & 0 & 0 \\
 0 & -2 & 1 & 0 & 0 & 0 \\
 -1 & 1 & -3 & 1 & 0 & 0 \\
 -4 & -2 & 3 & -4 & 1 & 0 \\
 -22 & -6 & -4 & 6 & -5 & 1 \\
\end{array}
\right).$$ 
Its inverse $\left(\sum_{n=0}^{\infty} n!x^n, x \sum_{n=0}^{\infty} n!x^n\right)$ begins
$$\left(
\begin{array}{cccccc}
 1 & 0 & 0 & 0 & 0 & 0 \\
 1 & 1 & 0 & 0 & 0 & 0 \\
 2 & 2 & 1 & 0 & 0 & 0 \\
 6 & 5 & 3 & 1 & 0 & 0 \\
 24 & 16 & 9 & 4 & 1 & 0 \\
 120 & 64 & 31 & 14 & 5 & 1 \\
\end{array}
\right).$$ 
The closely related Riordan array $\left(1, x \sum_{n=0}^{\infty} n!x^n\right)$ is \seqnum{A084938}. 
The unsigned sequence $$1,1, 0, 1, 4, 22, 144, 1089, 9308, 88562,\ldots$$ counts the number of generators in arity $n$ of the operad \emph{Lie}, when considered as a free non-symmetric operad \cite{Salvatore}.
\end{example}

\section{Duality and Riordan involutions}
We say that a Riordan array $(g(x), f(x))$ is self-dual if we have
$$(g(x), f(x))^! = (g(x), f(x)).$$ 
\begin{example} We consider the Riordan array $\left(-\frac{1}{1+x}, \frac{x}{1+x}\right)$ which begins
$$\left(
\begin{array}{cccccc}
 -1 & 0 & 0 & 0 & 0 & 0 \\
 1 & -1 & 0 & 0 & 0 & 0 \\
 -1 & 2 & -1 & 0 & 0 & 0 \\
 1 & -3 & 3 & -1 & 0 & 0 \\
 -1 & 4 & -6 & 4 & -1 & 0 \\
 1 & -5 & 10 & -10 & 5 & -1 \\
\end{array}
\right).$$ 
Its generating function is given by 
$$\frac{-\frac{1}{1+x}}{1-y \frac{x}{1+x}}=\frac{-1}{1-x(y-1)}.$$ 
In order to find the required inversion, we must therefore solve the equation 
$$\frac{-u}{1-u(y-1)}=x.$$ 
The solution gives us 
$$\frac{u}{x}=\frac{-1}{1-x(y-1)}.$$
Thus we have 
$$\left(-\frac{1}{1+x}, \frac{x}{1+x}\right)^!=\left(-\frac{1}{1+x}, \frac{x}{1+x}\right).$$ 
The Riordan array $\left(-\frac{1}{1+x}, \frac{x}{1+x}\right)$ is an example of a Riordan array that is self-dual.
\end{example}
\begin{example} Our next example concerns the \emph{cubical trialgebra operad} \cite{Loday}. The generating series of this operad is the generating series of the family of cubes $\frac{-x}{1+(y+2)x}$. This is $x$ times the  generating function of the Riordan array $\left(\frac{-1}{1+2x}, \frac{-x}{1+2x}\right)$. In order to find its inversion, we solve the equation
$$\frac{-u}{1+(y+2)u}=x,$$ to find that 
$$u=\frac{-x}{1+(y+2)x}.$$
Thus we have that 
$$\left(\frac{-1}{1+2x}, \frac{-x}{1+2x}\right)^!=\left(\frac{-1}{1+2x}, \frac{-x}{1+2x}\right).$$ 
That is, the Riordan array $\left(\frac{-1}{1+2x}, \frac{-x}{1+2x}\right)$ is self-dual. It is also an involution in the group of Riordan arrays: we have 
$$\left(\frac{-1}{1+2x}, \frac{-x}{1+2x}\right)^2=I=(1,x).$$
In general, we have 
$$\left(\frac{-1}{1+rx}, \frac{-x}{1+rx}\right)^!=\left(\frac{-1}{1+rx}, \frac{-x}{1+rx}\right),$$ since the solution of 
$$\frac{-u}{1+(y+r)u}=x,$$ with $u(0)=0$ is given by
$$u=\frac{-x}{1+(y+r)x}.$$
These matrices are again involutions in the group of Riordan arrays:
$$\left(\frac{-1}{1+rx}, \frac{-x}{1+rx}\right)^2=I.$$
\end{example}
\section{Inversions of one-parameter families}
In this section we investigate the inversions of elements of one-parameter families of Riordan arrays.
\begin{example} \textbf{The Riordan arrays} $\mathbf{(1+rx, x)}$. The general element of this family begins
$$\left(
\begin{array}{cccccc}
 1 & 0 & 0 & 0 & 0 & 0 \\
 r & 1 & 0 & 0 & 0 & 0 \\
 0 & r & 1 & 0 & 0 & 0 \\
 0 & 0 & r & 1 & 0 & 0 \\
 0 & 0 & 0 & r & 1 & 0 \\
 0 & 0 & 0 & 0 & r & 1 \\
\end{array}
\right).$$
This array has generating function $\frac{(1+rx)}{1-xy}$ from which we deduce that the generating function of the inversion is given by
$$\frac{\sqrt{1+2(y+2r)x+x^2y^2}-xy-1}{2rx}.$$
This expands to give the arrays that begin
$$\left(
\begin{array}{cccccc}
 1 & 0 & 0 & 0 & 0 & 0 \\
 -r & -1 & 0 & 0 & 0 & 0 \\
 2 r^2 & 3 r & 1 & 0 & 0 & 0 \\
 -5 r^3 & -10 r^2 & -6 r & -1 & 0 & 0 \\
 14 r^4 & 35 r^3 & 30 r^2 & 10 r & 1 & 0 \\
 -42 r^5 & -126 r^4 & -140 r^3 & -70 r^2 & -15 r & -1 \\
\end{array}
\right).$$
In order to find the general $(n,k)$-th element of this matrix, we take $g(x)=1+rx$ and $f(x)=x$ in the formula
$$\frac{(-1)^k}{n+1}\binom{n+1}{k}[x^n] f(x)^k \left(\frac{1}{g(x)}\right)^{n+1}.$$
We have
\begin{align*}
[x^n] f(x)^k \left(\frac{1}{g(x)}\right)^{n+1}&=[x^n]x^k \left(\frac{1}{1+rx}\right)^{n+1}\\
&=[x^{n-k}] \sum_{j=0}^{\infty} \binom{-(n+1)}{j}r^j x^j\\
&=[x^{n-k}] \sum_{j=0}^{\infty} \binom{n+j}{j}(-r)^j x^j\\
&=\binom{n+n-k}{n-k}(-r)^{n-j}\\
&=\binom{2n-k}{n-k}(-r)^{n-k}.\end{align*}
Thus we obtain
$$\hat{t}_{n,k}=\frac{(-1)^k}{n+1} \binom{n+1}{k}\binom{2n-k}{n-k}(-r)^{n-k}.$$ We therefore have 
$$(1+rx, x)^!=\left(\frac{(-1)^k}{n+1} \binom{n+1}{k}\binom{2n-k}{n-k}(-r)^{n-k}\right).$$ 
The matrix $((\Rev(xg(x))', \Rev(xg(x)))$ in this case is the matrix
$$\left(\frac{1}{\sqrt{1+4rx}}, \frac{\sqrt{1+4rx}-1}{2r}\right).$$
This therefore has general term $\binom{2n-k}{n-k}(-r)^{n-k}$ and begins
$$\left(
\begin{array}{ccccc}
 1 & 0 & 0 & 0 & 0 \\
 -2 r & 1 & 0 & 0 & 0 \\
 6 r^2 & -3 r & 1 & 0 & 0 \\
 -20 r^3 & 10 r^2 & -4 r & 1 & 0 \\
 70 r^4 & -35 r^3 & 15 r^2 & -5 r & 1 \\
\end{array}
\right).$$
The row sums of the inversion $\sum_{k=0}^n \frac{(-1)^k}{n+1} \binom{n+1}{k}\binom{2n-k}{n-k}(-r)^{n-k}$, which begin $$1, -r - 1, 2r^2 + 3r + 1, - 5r^3 - 10r^2 - 6r - 1, 14r^4 + 35r^3 + 30r^2 + 10r + 1,\ldots$$
are the revert transform of the row sums of $(1+rx, x)$, or
$$1,r+1,r+1, r+1, \ldots.$$
The generating function of the inversion may be expressed as the continued fraction
$$\cfrac{1}{1+(y+r)x-\cfrac{r(y+r)x^2}{1+(y+2r)x-\cfrac{r(y+r)x}{1+(y+2r)x-\cdots}}}.$$
The unsigned triangle $\frac{1}{n+1} \binom{n+1}{k}\binom{2n-k}{n-k}$, which begins 
$$\left(
\begin{array}{cccccc}
 1 & 0 & 0 & 0 & 0 & 0 \\
 1 & 1 & 0 & 0 & 0 & 0 \\
 2 & 3 & 1 & 0 & 0 & 0 \\
 5 & 10 & 6 & 1 & 0 & 0 \\
 14 & 35 & 30 & 10 & 1 & 0 \\
 42 & 126 & 140 & 70 & 15 & 1 \\
\end{array}
\right),$$ 
counts Schroeder paths from $(0,0)$ to $(2n,0)$ having $k$-peaks. This is \seqnum{A060693}. It is the inversion of the Riordan array $(1-x,-x)$ which begins
$$\left(
\begin{array}{cccccc}
 1 & 0 & 0 & 0 & 0 & 0 \\
 -1 & -1 & 0 & 0 & 0 & 0 \\
 0 & 1 & 1 & 0 & 0 & 0 \\
 0 & 0 & -1 & -1 & 0 & 0 \\
 0 & 0 & 0 & 1 & 1 & 0 \\
 0 & 0 & 0 & 0 & -1 & -1 \\
\end{array}
\right).$$ 
It follows that the matrix \seqnum{A088617}
$$\left(
\begin{array}{cccccc}
 1 & 0 & 0 & 0 & 0 & 0 \\
 1 & 1 & 0 & 0 & 0 & 0 \\
 1 & 3 & 2 & 0 & 0 & 0 \\
 1 & 6 & 10 & 5 & 0 & 0 \\
 1 & 10 & 30 & 35 & 14 & 0 \\
 1 & 15 & 70 & 140 & 126 & 42 \\
\end{array}
\right),$$ which counts Schroeder paths from $(0,0)$ to $(2n,0)$ with $k$ up-steps $U=(1,1)$ is the inversion of the triangle that begins 
$$\left(
\begin{array}{cccccc}
 1 & 0 & 0 & 0 & 0 & 0 \\
 -1 & -1 & 0 & 0 & 0 & 0 \\
 1 & 1 & 0 & 0 & 0 & 0 \\
 -1 & -1 & 0 & 0 & 0 & 0 \\
 1 & 1 & 0 & 0 & 0 & 0 \\
 -1 & -1 & 0 & 0 & 0 & 0 \\
\end{array}
\right).$$ 
\end{example}
\begin{example} We now turn to examine the inversions of the matrices of the form $\left(1-\frac{rx}{1+x}, -x\right)=\left(\frac{1-x(r-1)}{1+x},-x\right)$, which begin
$$\left(
\begin{array}{cccccc}
 1 & 0 & 0 & 0 & 0 & 0 \\
 -r & -1 & 0 & 0 & 0 & 0 \\
 r & r & 1 & 0 & 0 & 0 \\
 -r & -r & -r & -1 & 0 & 0 \\
 r & r & r & r & 1 & 0 \\
 -r & -r & -r & -r & -r & -1 \\
\end{array}
\right).$$
The inversion may be found by solving the equation 
$$\frac{u(1-u(r-1))}{(1+u)(1+uy)}=x$$ to get 
$$\frac{u}{x}=\frac{1-(1+y)x-\sqrt{1-2(y+2r-1)x-(1-y)^2x^2}}{2x(xy+r-1)}.$$ 
We obtain the family of matrices that begin
$$\left(
\begin{array}{ccccc}
 1 & 0 & 0 & 0 & 0 \\
 r & 1 & 0 & 0 & 0 \\
 r (2 r-1) & 3 r & 1 & 0 & 0 \\
 r \left(5 r^2-5 r+1\right) & 2 r (5 r-2) & 6 r & 1 & 0 \\
 r \left(14 r^3-21 r^2+9 r-1\right) & 5 r \left(7 r^2-6 r+1\right) & 10 r (3 r-1) & 10 r & 1 \\
\end{array}
\right).$$ For $r=-2,\ldots,2$ we obtain the matrices
$$\left(
\begin{array}{ccccc}
 1 & 0 & 0 & 0 & 0 \\
 -2 & 1 & 0 & 0 & 0 \\
 10 & -6 & 1 & 0 & 0 \\
 -62 & 48 & -12 & 1 & 0 \\
 430 & -410 & 140 & -20 & 1 \\
\end{array}
\right),\left(
\begin{array}{ccccc}
 1 & 0 & 0 & 0 & 0 \\
 -1 & 1 & 0 & 0 & 0 \\
 3 & -3 & 1 & 0 & 0 \\
 -11 & 14 & -6 & 1 & 0 \\
 45 & -70 & 40 & -10 & 1 \\
\end{array}
\right),\left(
\begin{array}{ccccc}
 1 & 0 & 0 & 0 & 0 \\
 0 & 1 & 0 & 0 & 0 \\
 0 & 0 & 1 & 0 & 0 \\
 0 & 0 & 0 & 1 & 0 \\
 0 & 0 & 0 & 0 & 1 \\
\end{array}
\right),$$
$$\left(
\begin{array}{ccccc}
 1 & 0 & 0 & 0 & 0 \\
 1 & 1 & 0 & 0 & 0 \\
 1 & 3 & 1 & 0 & 0 \\
 1 & 6 & 6 & 1 & 0 \\
 1 & 10 & 20 & 10 & 1 \\
\end{array}
\right), \left(
\begin{array}{ccccc}
 1 & 0 & 0 & 0 & 0 \\
 2 & 1 & 0 & 0 & 0 \\
 6 & 6 & 1 & 0 & 0 \\
 22 & 32 & 12 & 1 & 0 \\
 90 & 170 & 100 & 20 & 1 \\
\end{array}
\right).$$
We find that the general $(n,k)$-th term of the inversion matrix is given by
$$\hat{t}_{n,k}=\frac{1}{n+1} \binom{n+1}{k} \sum_{j=0}^{n+1} \binom{n+1}{j}\binom{2n-k-j}{n-k-j}(r-1)^{n-k-j}.$$ The bivariate generating function of the inversion matrix may be expressed as the continued fraction 
$$\cfrac{1}{1-(y+r)x-\cfrac{r(y+r-1)x^2}{1-(y+2r-1)x-\cfrac{r(y+r-1)}{1-(y+2r-1)x-\cdots}}}.$$
The generating function of the initial column of the inversion matrix is given by 
$$\frac{\sqrt{(1+x)^2-4rx}+x-1}{2(1-r)x}=\frac{1}{1-x}c\left(\frac{(r-1)x}{(1-x)^2}\right),$$ where 
$$c(x)=\frac{1-\sqrt{1-4x}}{2x}$$ is the generating function of the Catalan numbers $C_n=\frac{1}{n+1}\binom{2n}{n}$ \seqnum{A000108}. 
The generating function of the row sums of the inversion is given by 
$$\frac{1-2x-\sqrt{1-4rx}}{2(x+r-1)x}.$$ 
This expands to give the sequence with general term 
$$\sum_{k=0}^n \frac{n-k+1}{n+1}\binom{n+k}{k}r^k.$$ 
The coefficient array $\frac{n-k+1}{n+1}\binom{n+k}{k}$ is the array that begins
$$\left(
\begin{array}{cccccc}
 1 & 0 & 0 & 0 & 0 & 0 \\
 1 & 1 & 0 & 0 & 0 & 0 \\
 1 & 2 & 2 & 0 & 0 & 0 \\
 1 & 3 & 5 & 5 & 0 & 0 \\
 1 & 4 & 9 & 14 & 14 & 0 \\
 1 & 5 & 14 & 28 & 42 & 42 \\
\end{array}
\right).$$ This is \seqnum{A009766}, which has many combinatorial interpretations.

The matrix $\left((\Rev(xg(x)))', \Rev(xg(x))\right)$ is the matrix 
$$\left(\frac{1-2r+x+\sqrt{1+2x(1-2r)+x^2}}{2(1-r)\sqrt{1+2x(1-2r)+x^2}}, \frac{\sqrt{1+2x(1-2r)x+x^2}+x-1}{2(1-r)}\right),$$ with general term 

$$\sum_{j=0}^{n+1} \binom{n+1}{j}\binom{2n-k-j}{n-k-j}(r-1)^{n-k-j}.$$  
When $r=2$, the row sums of this matrix gives the sequence that begins 
$$1, 5, 25, 129, 681, 3653, 19825, 108545, 598417, \ldots.$$ 
This counts the number of peaks in Schroeder paths (\seqnum{A002002}). 
\end{example}
\begin{example} \textbf{The Riordan array} $\mathbf{\left(\frac{1}{(1-x)^m},x \right)}$. In this example we calculate $\left(\frac{1}{(1-x)^m},x \right)^!$. We have 
\begin{align*}
\frac{(-1)^k}{n+1}\binom{n+1}{k}[x^n]f(x)^k \left(\frac{1}{g(x)}\right)^{n+1}&=\frac{(-1)^k}{n+1}\binom{n+1}{k}[x^n]x^k ((1-x)^m)^{n+1}\\
&=\frac{(-1)^k}{n+1}\binom{n+1}{k}[x^{n-k}](1-x)^{m(n+1)}\\
&=\frac{(-1)^k}{n+1}\binom{n+1}{k}[x^{n-k}] \sum_{j=0}^{(n+1)m} \binom{(n+1)m}{j}(-1)^j x^j\\
&=\frac{(-1)^k}{n+1}\binom{n+1}{k}\binom{(n+1)m}{n-k}(-1)^{n-k}\\
&=\frac{(-1)^n}{n+1}\binom{n+1}{k}\binom{(n+1)m}{n-k}.\end{align*}
Thus we have 
$$\left(\frac{1}{(1-x)^m},x \right)^!=\left(\frac{(-1)^n}{n+1}\binom{n+1}{k}\binom{(n+1)m}{n-k}\right).$$ 
The numbers $\frac{1}{n+1}\binom{n+1}{k}\binom{(n+1)m}{n-k}$ are the Fuss-Narayana numbers, coefficients of the Fuss-Narayana polynomials \cite{Armstrong, Kirillov}. 
\end{example}
The previous examples involved Riordan arrays for which $f(x)=x$ or $f(x)=-x$. We now look at a case where $f(x)$ is non-trivial.
\begin{example} In this example, we consider the family of Pascal-like triangles $\left(\frac{1}{1-x}, \frac{x(1+rx)}{1-x}\right)$. For instance, when $r=1$ we obtain the binomial matrix, while for $r=1$ we obtain the Delannoy triangle. We will actually work with the variant $\left(\frac{1}{1+x}, -\frac{x(1+rx)}{1+x}\right)$. In order to find the inversion of these matrices, we must therefore solve the equation
$$\frac{u}{1+(1+y)u+ryu^2}=x.$$ We thus obtain the generating function of the inversion as 
$$\frac{u}{x}=\frac{1-(1+y)x-\sqrt{1-2(1+y)x+(1+2y(1-2r)+y^2)x^2}}{2rx^2y}.$$ 
This can be equivalently represented by the continued fraction 
$$\cfrac{1}{1-(1+y)x-\cfrac{ryx^2}{1-(1+y)x-\cfrac{ryx^2}{1-(1+y)x-\cdots}}}.$$
This generating function expands to give the inversion triangles of this family, which begin 
$$\left(
\begin{array}{cccccc}
 1 & 0 & 0 & 0 & 0 & 0 \\
 1 & 1 & 0 & 0 & 0 & 0 \\
 1 & r+2 & 1 & 0 & 0 & 0 \\
 1 & 3(r+1) & 3 (r+1) & 1 & 0 & 0 \\
 1 & 2 (3 r+2) & 2 \left(r^2+6 r+3\right) & 2 (3 r+2) & 1 & 0 \\
 1 & 5(2r+1) & 10 \left(r^2+3 r+1\right) & 10 \left(r^2+3 r+1\right) & 5 (2 r+1) & 1 \\
\end{array}
\right).$$ 
For $r=-2,\ldots ,2$ we obtain the number triangles that begin as follows.
$$\left(
\begin{array}{cccccc}
 1 & 0 & 0 & 0 & 0 & 0 \\
 1 & 1 & 0 & 0 & 0 & 0 \\
 1 & 0 & 1 & 0 & 0 & 0 \\
 1 & -3 & -3 & 1 & 0 & 0 \\
 1 & -8 & -10 & -8 & 1 & 0 \\
 1 & -15 & -10 & -10 & -15 & 1 \\
\end{array}
\right),\left(
\begin{array}{cccccc}
 1 & 0 & 0 & 0 & 0 & 0 \\
 1 & 1 & 0 & 0 & 0 & 0 \\
 1 & 1 & 1 & 0 & 0 & 0 \\
 1 & 0 & 0 & 1 & 0 & 0 \\
 1 & -2 & -4 & -2 & 1 & 0 \\
 1 & -5 & -10 & -10 & -5 & 1 \\
\end{array}
\right),\left(
\begin{array}{cccccc}
 1 & 0 & 0 & 0 & 0 & 0 \\
 1 & 1 & 0 & 0 & 0 & 0 \\
 1 & 2 & 1 & 0 & 0 & 0 \\
 1 & 3 & 3 & 1 & 0 & 0 \\
 1 & 4 & 6 & 4 & 1 & 0 \\
 1 & 5 & 10 & 10 & 5 & 1 \\
\end{array}
\right),$$ 
$$\left(
\begin{array}{cccccc}
 1 & 0 & 0 & 0 & 0 & 0 \\
 1 & 1 & 0 & 0 & 0 & 0 \\
 1 & 3 & 1 & 0 & 0 & 0 \\
 1 & 6 & 6 & 1 & 0 & 0 \\
 1 & 10 & 20 & 10 & 1 & 0 \\
 1 & 15 & 50 & 50 & 15 & 1 \\
\end{array}
\right), \left(
\begin{array}{cccccc}
 1 & 0 & 0 & 0 & 0 & 0 \\
 1 & 1 & 0 & 0 & 0 & 0 \\
 1 & 4 & 1 & 0 & 0 & 0 \\
 1 & 9 & 9 & 1 & 0 & 0 \\
 1 & 16 & 38 & 16 & 1 & 0 \\
 1 & 25 & 110 & 110 & 25 & 1 \\
\end{array}
\right).$$ 
We can regard these arrays as Pascal-like generalizations of the Narayana triangle, which is the case $r=1$. It is interesting to see that the binomial matrix belongs to this one-parameter family. The row sums, which begin
$$1, 2, r + 4, 6r + 8, 2r^2 + 24r + 16, 20r^2 + 80r + 32, 5r^3 + 120r^2 + 240r + 64,\ldots,$$ 
have exponential generating function 
$$\frac{e^{2x} I_1(2 \sqrt{r}x)}{\sqrt{r}x}.$$ For $r=1$, we get $C_{n+1}$. For $r=5$, the sequence begins
$$ 1,  2 , 9 , 38 , 186 , 932,  4889,\ldots.$$ This is \seqnum{A249925}. Interestingly, this is equal to 
$\sum_{k=0}^n C_k C_{n-k} F_{k+1} F_{n-k+1}$, were $F_n$ are the Fibonacci numbers. In general, the row sums are given by 
$$\sum_{k=0}^{\lfloor \frac{n}{2} \rfloor}\binom{n}{2k}2^{n-2k}C_k.$$ The generating function of the row sums can be expressed as the following continue fraction.
$$\cfrac{1}{1-2x-\cfrac{rx^2}{1-2x-\cfrac{rx^2}{1-2x-\cdots}}}.$$
In order to find an expression for the general $(n,k)$-th term of the inversion triangle, we calculate the following. 
\begin{align*} [x^n] f(x)^k \left(\frac{1}{g(x)}\right)^{n+1}
&=[x^n]\left(\frac{-x(1+rx)}{1+x}\right)^k (1+x)^{n-k+1}\\
&=(-1)^k [x^{n-k}] (1+rx)^k (1+x)^{n-k+1}\\
&= (-1)^k [x^{n-k}] \sum_{j=0}^k \binom{k}{j}r^j x^j \sum_{i=0}^{n-k+1} \binom{n-k+1}{i}x^i\\
&= (-1)^k \sum_{j=0}^k \binom{k}{j}r^j \binom{n-k+1}{n-k-j}.\end{align*}
We thus obtain the general $(n,k)$-th element of the inversion to be 
$$\hat{t}_{n,k}=\frac{1}{n+1}\binom{n+1}{k}\sum_{j=0}^k \binom{k}{j}r^j \binom{n-k+1}{n-k-j}.$$ 
The inverse binomial transform of the inversion matrix is the matrix that begins 
$$\left(
\begin{array}{ccccccc}
 1 & 0 & 0 & 0 & 0 & 0 & 0 \\
 0 & 1 & 0 & 0 & 0 & 0 & 0 \\
 0 & r & 1 & 0 & 0 & 0 & 0 \\
 0 & 0 & 3 r & 1 & 0 & 0 & 0 \\
 0 & 0 & 2 r^2 & 6 r & 1 & 0 & 0 \\
 0 & 0 & 0 & 10 r^2 & 10 r & 1 & 0 \\
 0 & 0 & 0 & 5 r^3 & 30 r^2 & 15 r & 1 \\
\end{array}
\right),$$ with general term 
$$\frac{1}{k+1} \binom{n}{k}\binom{k+1}{n-k+1}.$$ 
The row sums of the second inverse binomial transform of the inversion matrix begin
$$1, 0, r, 0, 2r^2, 0, 5r^3, 0, 14r^4, 0, 42r^5,\ldots$$ making the link to the Catalan numbers explicit.
\end{example}
\begin{example} Our last example of this section calculates the inversions $(\hat{t}_{n,k})$ of the Riordan arrays
$$\left(\frac{1}{(1-x)^m}, \frac{x}{1-x}\right)=\left(\binom{n+m-1}{n-k}\right).$$
We have 
\begin{align*}
\frac{(-1)^k}{n+1}\binom{n+1}{k}[x^n]f(x)^k\left(\frac{1}{g(x)}\right)^{n+1}&=
\frac{(-1)^k}{n+1}\binom{n+1}{k}[x^n]\frac{x^k}{(1-x)^k} ((1-x)^m)^{n+1}\\
&=\frac{(-1)^k}{n+1}\binom{n+1}{k}[x^{n-k}] (1-x)^{m(n+1)-k}\\
&=\frac{(-1)^k}{n+1}\binom{n+1}{k}[x^{n-k}] \sum_{j=0}^{\infty}\binom{m(n+1)-k}{j}(-1)^j x^j\\
&=\frac{(-1)^k}{n+1}\binom{n+1}{k} \binom{m(n+1)-k}{n-k}(-1)^{n-k}\\
&=\frac{(-1)^n}{n+1}\binom{n+1}{k} \binom{m(n+1)-k}{n-k}.\end{align*}
Thus we have
$$\left(\frac{1}{(1-x)^m}, \frac{x}{1-x}\right)^!=\left(\binom{n+m-1}{n-k}\right)^!=\left(\frac{(-1)^n}{n+1}\binom{n+1}{k} \binom{m(n+1)-k}{n-k}\right).$$
We can also deduce  that 
$$\left(\binom{m(n+1)-k}{n-k}\right)=\left(\left(\Rev\left(\frac{x}{(1-x)^m}\right)\right)', \frac{\Rev\left(\frac{x}{(1-x)^m}\right)}{1-\Rev\left(\frac{x}{(1-x)^m}\right)}\right).$$ 
In particular, we obtain
$$\binom{m(n+1)}{n}=[x^n]\left(\Rev\left(\frac{x}{(1-x)^m}\right)\right)'=(n+1)[x^{n+1}]\Rev\left(\frac{x}{(1-x)^m}\right),$$ or
$$\frac{1}{n+1} \binom{m(n+1)}{n} = [x^{n+1}]\Rev\left(\frac{x}{(1-x)^m}\right).$$ 
The following table documents these inversions for $m=-1, \ldots, 4$ \cite{Novelli}. 
\begin{center} 
\begin{tabular}{|c|c|c|}\hline
$m$ &$\left(\binom{n+m-1}{n-k}\right)$ & $\left(\binom{n+m-1}{n-k}\right)^!=\left(\frac{(-1)^n}{n+1}\binom{n+1}{k} \binom{m(n+1)-k}{n-k}\right)$\\\hline
$-1$ & $\left(
\begin{array}{ccccc}
 1 & 0 & 0 & 0 & 0 \\
 -1 & 1 & 0 & 0 & 0 \\
 0 & 0 & 1 & 0 & 0 \\
 0 & 0 & 1 & 1 & 0 \\
 0 & 0 & 1 & 2 & 1 \\
\end{array}
\right)$ & $\left(
\begin{array}{ccccc}
 1 & 0 & 0 & 0 & 0 \\
 1 & -1 & 0 & 0 & 0 \\
 2 & -4 & 1 & 0 & 0 \\
 5 & -15 & 9 & -1 & 0 \\
 14 & -56 & 56 & -16 & 1 \\
\end{array}
\right)$\\\hline
$0$ & $\left(
\begin{array}{ccccc}
 1 & 0 & 0 & 0 & 0 \\
 0 & 1 & 0 & 0 & 0 \\
 0 & 1 & 1 & 0 & 0 \\
 0 & 1 & 2 & 1 & 0 \\
 0 & 1 & 3 & 3 & 1 \\
\end{array}
\right)$ & $\left(
\begin{array}{ccccc}
 1 & 0 & 0 & 0 & 0 \\
 0 & -1 & 0 & 0 & 0 \\
 0 & -1 & 1 & 0 & 0 \\
 0 & -1 & 3 & -1 & 0 \\
 0 & -1 & 6 & -6 & 1 \\
\end{array}
\right)$ \\\hline
$1$ & $\left(
\begin{array}{ccccc}
 1 & 0 & 0 & 0 & 0 \\
 1 & 1 & 0 & 0 & 0 \\
 1 & 2 & 1 & 0 & 0 \\
 1 & 3 & 3 & 1 & 0 \\
 1 & 4 & 6 & 4 & 1 \\
\end{array}
\right)$ & $\left(
\begin{array}{ccccc}
 1 & 0 & 0 & 0 & 0 \\
 -1 & -1 & 0 & 0 & 0 \\
 1 & 2 & 1 & 0 & 0 \\
 -1 & -3 & -3 & -1 & 0 \\
 1 & 4 & 6 & 4 & 1 \\
\end{array}
\right)$ \\\hline
$2$ & $\left(
\begin{array}{ccccc}
 1 & 0 & 0 & 0 & 0 \\
 2 & 1 & 0 & 0 & 0 \\
 3 & 3 & 1 & 0 & 0 \\
 4 & 6 & 4 & 1 & 0 \\
 5 & 10 & 10 & 5 & 1 \\
\end{array}
\right)$ & $\left(
\begin{array}{ccccc}
 1 & 0 & 0 & 0 & 0 \\
 -2 & -1 & 0 & 0 & 0 \\
 5 & 5 & 1 & 0 & 0 \\
 -14 & -21 & -9 & -1 & 0 \\
 42 & 84 & 56 & 14 & 1 \\
\end{array}
\right)$  \\\hline
$3$ & $\left(
\begin{array}{ccccc}
 1 & 0 & 0 & 0 & 0 \\
 3 & 1 & 0 & 0 & 0 \\
 6 & 4 & 1 & 0 & 0 \\
 10 & 10 & 5 & 1 & 0 \\
 15 & 20 & 15 & 6 & 1 \\
\end{array}
\right)$ & $\left(
\begin{array}{ccccc}
 1 & 0 & 0 & 0 & 0 \\
 -3 & -1 & 0 & 0 & 0 \\
 12 & 8 & 1 & 0 & 0 \\
 -55 & -55 & -15 & -1 & 0 \\
 273 & 364 & 156 & 24 & 1 \\
\end{array}
\right)$ \\\hline
$4$ & $\left(
\begin{array}{ccccc}
 1 & 0 & 0 & 0 & 0 \\
 4 & 1 & 0 & 0 & 0 \\
 10 & 5 & 1 & 0 & 0 \\
 20 & 15 & 6 & 1 & 0 \\
 35 & 35 & 21 & 7 & 1 \\
\end{array}
\right)$ & $\left(
\begin{array}{ccccc}
 1 & 0 & 0 & 0 & 0 \\
 -4 & -1 & 0 & 0 & 0 \\
 22 & 11 & 1 & 0 & 0 \\
 -140 & -105 & -21 & -1 & 0 \\
 969 & 969 & 306 & 34 & 1 \\
\end{array}
\right)$ \\\hline
\end{tabular}
\end{center}
The last two triangles on the right are signed versions of \seqnum{A243662} and \seqnum{A243663} \cite{Novelli}.
\end{example}

\section{The inversion of exponential Riordan arrays}
In this section we give a brief introduction to the inversion of exponential Riordan arrays. We begin the section by looking at the example of the binomial matrix, but this time, regarded as the exponential Riordan array $[e^x, x]$.
\begin{example} We recall that the exponential Riordan array $[e^x, x]$ has general element given by $\binom{n}{k}$. This is so because we have
\begin{align*}
\frac{n!}{k!}[x^n]e^x x^k&= \frac{n!}{k!} [x^{n-k}] \sum_{i=0}^{\infty} \frac{x^i}{i!}\\
&= \frac{n!}{k!} \frac{1}{(n-k)!}\\
&= \binom{n}{k}.\end{align*}
By the theory of exponential Riordan arrays, the bivariate generating function of this array is given by
$$G_e(x,y)=e^x e^{xy}= e^{x+xy} = e^{x(1+y)}.$$
The revert transform (in $x$) of $G_e(x,y)$ is given by $\frac{d}{dx} \Rev\left(\int_0^x G_e(t,y)\,dt\right)$.
We have that
$$\int_0^x e^{t(1+y)} = \frac{e^{x(1+y)-1}}{1+y}.$$
To carry out the inversion, we must then solve the equation
$$\frac{e^{u(1+y)-1}}{1+y}=x$$ to find the solution $u(x)$ that satisfies $u(0)=0$.
We find that
$$u=\frac{ln(1+x(1+y))}{1+y}.$$
We then differentiate this (with respect to $x$) to get the revert transform of $G_e(x,y)$.
We find that
$$\hat{G}_e(x,y)=\frac{1}{1+x(1+y)}.$$
This is the generating function of the array that begins
$$\left(
\begin{array}{cccccc}
 1 & 0 & 0 & 0 & 0 & 0 \\
 -1 & -1 & 0 & 0 & 0 & 0 \\
 2 & 4 & 2 & 0 & 0 & 0 \\
 -6 & -18 & -18 & -6 & 0 & 0 \\
 24 & 96 & 144 & 96 & 24 & 0 \\
 -120 & -600 & -1200 & -1200 & -600 & -120 \\
\end{array}
\right).$$
This is the inversion of the exponential Riordan array $[e^x, x]$. The general element of this matrix is $(-1)^n n! \binom{n}{k}$. The triangle is a signed version of \seqnum{A196347}.
\end{example}
\begin{example} We consider the inversion of the exponential Riordan array $[\cosh(x), x]$. The array $[\cosh(x), x]$ \seqnum{A119467} begins
$$\left(
\begin{array}{ccccccc}
 1 & 0 & 0 & 0 & 0 & 0 & 0 \\
 0 & 1 & 0 & 0 & 0 & 0 & 0 \\
 1 & 0 & 1 & 0 & 0 & 0 & 0 \\
 0 & 3 & 0 & 1 & 0 & 0 & 0 \\
 1 & 0 & 6 & 0 & 1 & 0 & 0 \\
 0 & 5 & 0 & 10 & 0 & 1 & 0 \\
 1 & 0 & 15 & 0 & 15 & 0 & 1 \\
\end{array}
\right).$$ We have $G_e(x,y)= \cosh(x) e^{xy}$, and
$$ \int_0^x G_e(t,y)\, dt = e^{xy}\left(\frac{e^x}{2(1+y)}-\frac{e^{-x}}{2(1-y)}\right)+\frac{y}{1-y^2}.$$
Unfortunately, there is no closed solution to the equation
$$e^{uy}\left(\frac{e^u}{2(1+y)}-\frac{e^{-u}}{2(1-y)}\right)+\frac{y}{1-y^2}=x.$$
Nevertheless, we can calculate as many rows of the inversion of the above matrix as we like by calculating the first column of the inverse of the exponential Riordan array
$$\left[1, e^{xy}\left(\frac{e^x}{2(1+y)}-\frac{e^{-x}}{2(1-y)}\right)+\frac{y}{1-y^2}\right].$$
For instance, we have
$$\left(
\begin{array}{ccccc}
 1 & 0 & 0 & 0 & 0 \\
 0 & 1 & 0 & 0 & 0 \\
 0 & y & 1 & 0 & 0 \\
 0 & y^2+1 & 3 y & 1 & 0 \\
 0 & y \left(y^2+3\right) & 7 y^2+4 & 6 y & 1 \\
\end{array}
\right)^{-1}=\left(
\begin{array}{ccccc}
 1 & 0 & 0 & 0 & 0 \\
 0 & 1 & 0 & 0 & 0 \\
 0 & -y & 1 & 0 & 0 \\
 0 & 2 y^2-1 & -3 y & 1 & 0 \\
 0 & y \left(7-6 y^2\right) & 11 y^2-4 & -6 y & 1 \\
\end{array}
\right).$$
Thus we expand the polynomial sequence $1, -y, 2y^2-1, y(7-6y^2), \ldots$ to obtain the rows of the inversion of $[cosh(x), x]$. Thus the matrix $[\cosh(x), x]^!$  begins
$$\left(
\begin{array}{ccccccc}
 1 & 0 & 0 & 0 & 0 & 0 & 0 \\
 0 & -1 & 0 & 0 & 0 & 0 & 0 \\
 -1 & 0 & 2 & 0 & 0 & 0 & 0 \\
 0 & 7 & 0 & -6 & 0 & 0 & 0 \\
 9 & 0 & -46 & 0 & 24 & 0 & 0 \\
 0 & -159 & 0 & 326 & 0 & -120 & 0 \\
 -225 & 0 & 2134 & 0 & -2556 & 0 & 720 \\
\end{array}
\right).$$
The first column
$$1,0,-1,0,9,0,-225,\ldots$$ is then the exponential revert transform of the sequence
$$1,0,1,0,1,0,1,0,\ldots$$ which is the expansion of $\cosh(x)$. The numbers occurring in this reversion are the squares of the double factorial numbers.

The row sums of the inversion, which begin
$$1, -1, 1, 1, -13, 47, 73,\ldots,$$ are the exponential revert transform of the row sums
$$1,1,2,4,8,16,\ldots$$ of $[\cosh(x), x]$. This latter sequence has exponential generating function $e^x \cosh(x)$. To find the revert transform of this, we proceed as follows.
First, we calculate
$$\int_0^x e^t \cosh(t)\,dt=\frac{1}{4}\left(e^{2x}+2x-1\right).$$ We then solve the equation
$$\frac{1}{4}\left(e^{2u}+2u-1\right)=x$$ to obtain the solution $u(x)$ such that $u(0)=0$.
We obtain
$$u(x)=\frac{1}{2} \left(-W\left(e^{4 x+1}\right)+4 x+1\right).$$
The revert transform we seek is then the derivative of this. Thus the rows sums
$$1, -1, 1, 1, -13, 47, 73, -2447, 16811, 15551, -1726511,\ldots$$ of the inversion of $[\cosh(x), x]$ have generating function
$$\frac{1}{2} \left(4-\frac{4 W\left(e^{4 x+1}\right)}{W\left(e^{4 x+1}\right)+1}\right).$$
These terms coincide with the coefficients of Airey's converging factor \seqnum{A001662} \cite{Airey}.
Sergei N. Gladkovskii gives the following continued fraction expression for the generating function of the revert transform of $e^x \cosh(x)$:
$$\cfrac{1}{1+\cfrac{x}{1-2x+\cfrac{2x}{1-4x+\cfrac{3x}{1-6x+\cfrac{4x}{1-8x+\cdots}}}}}.$$ 
\end{example}

\section{Conclusions} In this note, we have defined the notion of the inversion of a Riordan array, along with methods for its construction. Examples show that the process of going from a Riordan array to its inversion leads to many number triangles and sequences of combinatorial importance. In the reverse direction, it is probable that many number triangles of combinatorial significance are the inversions of Riordan arrays, or arrays closely related to Riordan arrays. It is hoped that this interplay between Riordan arrays and their inversions will provide extra insight in the context of algebraic combinatorics and related fields.

\bigskip
\hrule
\bigskip
\noindent 2010 {\it Mathematics Subject Classification}: Primary
15A30; Secondary 15B36, 05A10, 13F25.
\noindent \emph{Keywords:} Riordan array, Riordan group, invert transform, revert transform, inversion, Lagrange inversion, associahedron, Narayana numbers, Fuss-Narayana polynomials.

\bigskip
\hrule
\bigskip
\noindent (Concerned with sequences
\seqnum{A000108},
\seqnum{A001263},
\seqnum{A001662},
\seqnum{A002002},
\seqnum{A007318},
\seqnum{A009766},
\seqnum{A049310},
\seqnum{A053121},
\seqnum{A060693},
\seqnum{A084938},
\seqnum{A088617},
\seqnum{A094587},
\seqnum{A097610},
\seqnum{A119467},
\seqnum{A126216},
\seqnum{A196347},
\seqnum{A243662},
\seqnum{A243663}, and
\seqnum{A249925}.)

\end{document}